\documentclass[12pt]{article}
\usepackage{url}
\usepackage{amssymb,latexsym,amsmath,enumerate,verbatim,amsfonts,amsthm}
\usepackage{color} 
\usepackage[active]{srcltx}
\newtheorem{theorem}{Theorem}

\newtheorem{remark}{Remark}

\def\R{{\rm I\!R}}

\def\argmin{\mathop{\rm arg\,min}}
\def\argmax{\mathop{\rm arg\,max}}
\newcommand{\e}{{\bf e}}
\newcommand{\h}{{\bf h}}
\renewcommand{\t}{{\bf t}}
\renewcommand{\u}{{\bf u}}
\renewcommand{\v}{{\bf v}}
\newcommand{\x}{{\bf x}}
\newcommand{\y}{{\bf y}}
\newcommand{\blambda}{\mbox{\boldmath{$\lambda$}}}
\newcommand{\bmu}{\mbox{\boldmath{$\mu$}}}
\newcommand{\bnu}{\mbox{\boldmath{$\nu$}}}
\newcommand{\bz}{{\bf 0}}

\title{Convex relaxation for finding planted influential nodes
in a social network\thanks{Supported
in part by a grant from the U.~S.~Air Force Office of Scientific Research and in part
by a Discovery Grant from the Natural Sciences and Engineering Research Council
(NSERC) of Canada.}}
\author{Lisa Elkin\thanks{Department of
Combinatorics and Optimization, University of Waterloo, 200 University Ave.~W.,
Waterloo, Ontario, Canada, N2L3G1, {\tt laelkin@uwaterloo.ca}.}
\and
Ting Kei Pong\thanks{Department of
Combinatorics and Optimization, University of Waterloo, 200 University Ave.~W.,
Waterloo, Ontario, Canada, N2L3G1, {\tt tkpong@gmail.com}.}
\and
Stephen A.~Vavasis\thanks{Corresponding author.
Department of
Combinatorics and Optimization, University of Waterloo, 200 University Ave.~W.,
Waterloo, Ontario, Canada, N2L3G1, {\tt vavasis@uwaterloo.ca}.}}

\begin{document}
\maketitle
\begin{abstract}
We consider the problem of maximizing influence in a social network.
We focus on the case that the social network is a directed bipartite
graph whose arcs join senders
to receivers.
We consider both the case of deterministic networks
and probabilistic graphical models, that is, the so-called ``cascade'' model.
The problem is to find the set of the $k$ most influential
senders for a given integer $k$.
Although this problem is NP-hard, there is a polynomial-time
approximation algorithm due to Kempe, Kleinberg and Tardos.  In this work we
consider convex relaxation for the problem.  We prove that convex optimization
can recover the exact optimizer
in the case that the network is constructed
according to a generative model in which influential nodes are planted but then
obscured with noise.   We also demonstrate
computationally that the convex relaxation can succeed on a more realistic generative
model called the ``forest fire'' model.
\end{abstract}

\section{Influence in social networks}
The formation and growth of vast on-line social networks
in the past decade has fueled substantial research into the problem
of identifying influential members in these networks.
An obvious application is determining how
to quickly spread an urgent message over a social network.
Another obvious
application of this research is to determine optimal members of a social network
for advertisers to target.
Social network research has
also been applied to model the spread of health problems
by epidemiologists \cite{Christakis}, in which
case influential nodes would correspond to the persons most in need
of medical intervention.

For the purpose
of this work, we regard a social network as a directed graph.  An arc
represents a communication link between a sender and receiver.
In the case of a general directed graphs, nodes can be both
senders (outdegree $\ge 1$) and receivers (indegree $\ge 1$).
The network passes through a discrete sequence of states.  At
each discrete state, some of the nodes possess a message.  When the network advances
to its next state,  a node with a message may pass this message along outgoing arcs
according to a postulated model for message transmission.
In this work, we will consider a deterministic model and a
probabilistic model.  With these specifications in place, it is now
possible to pose the question of finding the $k$ most influential
nodes in the network.  In other words, given an integer $k$, find the
subset of $k$ nodes such that if a message is seeded at those nodes,
the largest number (or expected largest number) of receivers are
eventually reached at later time steps.

This problem was first investigated in depth in an
influential 2003 paper by Kempe, Kleinberg and
Tardos \cite {Kempe}.  They note
that the problem is NP-hard.  Their main contribution is a
polynomial-time algorithm based on sampling and on the greedy method
for finding an approximate solution to the maximization problem that
is within 63\% of optimum.  Their algorithm is valid for two different
probabilistic communication models.

We adopt the following point of view.
We specialize to the bipartite case, that is,
the graph consists of distinct senders (indegree $= 0$)
and receivers (outdegree $ = 0$) connected by a single layer of arcs.
As we explain below, in the deterministic case, this specialization is
without loss of generality.  Within this framework, we propose a generative
model in which the senders are either planted influencers or subordinates and
the receivers are partitioned into interest groups.  The network
of influencers, subordinates and interest groups is, however, partly
obscured by noise arcs.

We then propose a natural convex relaxation for the problem.  We show that
the convex relaxation is able to recover the planted influencers with
high probability assuming the parameters of the generative model lie in
certain ranges.  We have separate results for the deterministic graph
model and the probabilistic graph model, known as the ``independent cascade''
model.

This line of analysis fits into a recent body of results showing that
many NP-hard problems can be solved in polynomial time using convex
relaxation assuming the data is generated in a certain way.  A notable
pioneering work in this regard was the discovery of ``compressive
sensing'' by Donoho \cite{donoho-underdet} and by Cand\`es and Tao
\cite{CandesTao2}.  This line of attack has also been used to analyze
problems in data mining including the clustering problem
\cite{AmesVavasis2, Ames} and nonnegative matrix factorization
\cite{DoanVavasis}. The rationale for this line of work is that, although
the problems under consideration are NP-hard, it may still be possible
to solve them in polynomial time for `realistic' data, i.e., data
arising in real-world applications.  The reason is that realistic data
may possess properties that make the problem of finding hidden structure
more tractable
than in the case of data constructed by an adversary (as in an NP-hardness proof).  One
way to make progress in this regard is to postulate a generative model
for the data that attempts to capture some real-world characteristics.

In the next section, we focus on the deterministic case of the problem.
The more widely used (and presumably realistic) probabilistic model
is then analyzed in Section~\ref{sec:probmodel}.  Finally, in
Section~\ref{sec:forestfire} we consider the even more realistic ``forest
fire'' model of social networks.  We are not able to analyze this model, but we show
with computational results that the convex relaxation holds promise
for this case as well.

\section{Deterministic graph model}
In this section we postulate a deterministic model of a social network, that is,
each communication link passes messages from its tail to its head with probability
equal to 1.  In this model, the problem of influence
maximization is formally stated as follows.  Given a directed
graph $G=(V,E)$, and given an integer $k$, find a subset
$V^*\subset V$ such that $|V^*|=k$ and, subject to this constraint,
$|\Delta(V^*)|$ is maximum.  Here, $\Delta(V^*)$ denotes the subset of $V$
containing nodes
reachable by a directed path that begins from a node of $V^*$.

It is not hard to see that one can replace the original
network (an arbitrary directed graph) with a bipartite network.  In
particular, make two copies of each node (the `sender copy' and
the `receiver copy'), and put an arc $(i,j)$ in the resulting graph
whenever $i$ is the sender copy of an original node $i_0$, $j$ is
the receiver copy of original node $j_0$, and the original network has a directed
path (possibly of length 0 if $i_0=j_0$) from $i_0$ to
$j_0$.  This reduction to the bipartite case causes a blow-up of at
most quadratic size and hence does not affect the polynomial
solvability of the problem.

It is also easy to see that the bipartite deterministic case is essentially
equivalent to the classic set-cover problem, which is one of Garey and Johnson's
\cite{GJ79}
original NP-hard problems.  This shows that the problem of finding
the $k$ most influential nodes of a social
network, even in this apparently simplified case,
is NP-hard.

We now describe a particular class of bipartite deterministic
networks suitable for analysis.  Let the graph be
denoted
$G = (V_1,V_2,E)$.  The nodes of $V_1$, which are the {\em senders}, consist
of
{\em influencers} and {\em subordinates}.
The nodes in $V_2$ are called {\em receivers}.  All arcs in $E$ are directed
from $V_1$ to $V_2$.

We suppose that $V_1$ is partitioned into
$k$ disjoint {\em interest groups} $L_1,\ldots,L_k$,
each having a single influencer and $r_l\ge 0$ subordinates for $l = 1,\ldots,k$.
We suppose that $V_2$ is also partitioned into $k$
interest groups, say $V_2=G_1\cup\cdots\cup G_k$,
and let
$n_l=|G_l|$ for each $l$.
For notational convenience, we assume further
that the nodes in $V_1$ (resp., $V_2$) are arranged
according to the order of $L_1,\ldots,L_k$ (resp., $G_1,\ldots,G_k$),
and that within each interest group in $V_1$,
the first node is always the influencer.

We start by considering the following assumptions on the influencers and subordinates,
which corresponds to the noiseless case.  This is an easy case that will
clarify our assumptions and notation.
\begin{enumerate}
  \item[{\bf {A}1}] There is an arc
from the influencer in group $L_l$ to every receiver in $G_l$, $l=1,\ldots,k$.
  \item[{\bf {A}2}] There are no arcs outside interest group boundaries,
i.e., there is no arc from $L_l$ to $G_{l'}$ if $l\ne l'$.
  \item[{\bf {A}3}] Each subordinate in $L_l$ is adjacent to a proper
subset of $G_l$.
\end{enumerate}
It is readily apparent from these assumptions
that the solution to the problem of finding the $k$
most influential nodes is to take the $k$ influencers.

Let $A$ be the $|V_1|\times |V_2|$ matrix whose $(i,j)$th entry is $1$ if
there is an arc from the $i$th node in $V_1$ to
the $j$th node in $V_2$. Also, let $\x\in \R^{|V_1|}$ denote an
indicator vector for a node in $V_1$, and let $\x^*$
be the indicator vector corresponding to the influencers. Then it is clear from
assumptions ${\bf A1}$ through ${\bf A3}$ that the vector $\x^*$
is an optimal solution to the following integer programming problem:
\begin{equation*}
    \begin{array}{rl}
      \max\limits_{\x,\t} & \e^T\t\\
      {\rm s.t.} & \t \le A^T\x,\\
                 &  \bz \le \t\le \e,\\
                 & \e^T \x = k,\\
                 & \x \in \{0,1\}^{|V_1|},
    \end{array}
\end{equation*}
where $\e$ is the vector of all ones, with appropriate dimension.
This integer LP models the problem of
finding the $k$ most influential nodes.
The variable $\x$ contains a `$1$' entry for the
selected nodes
in $V_1$ and $0$ else.  The variable $\t$ can only be nonzero at a receiver
adjacent to a selected node, and is 0 else.  The objective is to maximize
the number of `1' entries in $\t$ subject to the constraint that only $k$
entries of $\x$ may be set to 1.  It is not hard to see that
vector $\t$ will always be integral
at the optimizer, so there is no need for an additional integrality constraint.

An equivalent version of the above maximization problem
maximizes the nonsmooth continuous function $\e^T(\min\{\e,A^T\x\})$ over the
discrete set $\Omega:=\{\x \in \{0,1\}^{|V_1|}:\; \e^T\x = k\}$.
Since $|\Omega| = \binom{n}{k}$, the above integer programming problem can be solved by a brute-force function evaluation
approach in polynomial time when $k = O(1)$.
However, this approach can be inefficient when $k$ is large.

Thus, we consider the following simple convex relaxation:
\begin{equation}\label{LP_basic}
    \begin{array}{rl}
      \max\limits_{\x,\t} & \e^T\t\\
      {\rm s.t.} & \t \le A^T\x,\\
                 & \bz\le \t\le \e,\\
                 & \e^T \x = k,\\
                 & \bz\le \x\le \e.
    \end{array}
  \end{equation}
We would like to study when \eqref{LP_basic} has $\x = \x^*$ as its unique solution.
In this case, the relaxation is tight and
the influencers can be identified by solving the linear program \eqref{LP_basic},
which can be solved by interior point methods in polynomial time.

We have the following result.
\begin{theorem}\label{thm1}
  Assume {\bf A1} through {\bf A3}. Then $(\x^*,\e)$
is the unique solution of \eqref{LP_basic}.
\end{theorem}
\begin{proof}
  We first prove that $(\x,\t) = (\x^*,\e)$ is optimal for \eqref{LP_basic}.
The argument that we now present for optimality is more complicated than
necessary, but the same argument will later establish uniqueness and
also be used in the more general case below.
We note first that the feasible set of \eqref{LP_basic} is
nonempty and compact
  and thus an optimal solution exists. Furthermore, a feasible solution $(\x,\t)$
of \eqref{LP_basic} is optimal if and only if there exist $(\blambda,\bmu,\bnu,\xi)$
  satisfying the following Karush-Kuhn-Tucker (KKT) conditions:
  \begin{align}\label{opt}
  \begin{array}{rcr}
    \blambda^T(\t - A^T \x)  = 0,      &&   \x^T(-A\blambda + \bnu + \xi \e)  = 0,\\
    \bmu^T(\e - \t)  =0,               && \blambda + \bmu  \ge \e,\\
    \bnu^T(\e - \x)  = 0,              && -A\blambda + \bnu + \xi \e  \ge \bz,\\
    \t^T(\blambda + \bmu - \e)  = 0,    && \blambda \ge \bz,  \bmu \ge \bz, \ \bnu\ge \bz.
  \end{array}
  \end{align}
  We shall show that $(\x,\t) = (\x^*,\e)$ is optimal by
exhibiting explicitly a quadruple $(\blambda^*,\bmu^*,\bnu^*,\xi^*)$ satisfying the above conditions.

  To proceed, set $\delta = (\max_{1\le l\le k}\{n_l\})^{-1}$ and define
  \begin{equation}\label{lambdamunuxi}
    \blambda^* = \begin{pmatrix}
      \e/n_1\\ \vdots \\ \e/n_k
    \end{pmatrix},\ \ \bmu^* = \e - \blambda^*, \ \ \bnu^* = \delta \x^* \ {\rm and}\
 \xi^* = 1 - \delta.
  \end{equation}
  Then for those $i$ such that $x^*_i>0$ (influencers), we have $(-A\blambda^* +
\bnu^* + \xi^* \e)_i = 0$, while for
  those $i$ such that $x^*_i = 0$ (subordinates), from assumption {\bf A3}, we have
  \[
  (-A\blambda^* + \bnu^* + \xi^* \e)_i \ge -1+\frac{1}{\max_{1\le l\le k}n_l} + 1-\delta = 0.
  \]
  From these and the definitions of $(\blambda^*,\bmu^*,\bnu^*,\xi^*)$,
it is routine to show that the conditions in \eqref{opt} are satisfied.
  Thus, $(\x^*,\e)$ is optimal for \eqref{LP_basic}.

  We now show that $(\x^*,\e)$ is the unique optimal solution
for \eqref{LP_basic}. Suppose that $(\x^\diamond,\t^\diamond)$ is an optimal solution
  for \eqref{LP_basic}. Since $\bz\le \t^\diamond \le \e$ and the
optimal value of \eqref{LP_basic} has to be $\e^T\e = \sum_{l=1}^kn_l$, it follows
  immediately that $\t^\diamond = \e$. Furthermore, from saddle point theory,
$(\x^\diamond,\t^\diamond)$ together with the $(\blambda^*,\bmu^*,\bnu^*,\xi^*)$
  constructed above has to satisfy \eqref{opt}. In particular, it follows
from ${\bnu^*}^T(\e-\x^\diamond) = 0$ that $\x^\diamond$ must
  equal 1 for the $k$ entries corresponding to influencers. This together with feasibility of $\x^\diamond$ gives $\x^\diamond = \x^*$.
  This completes the proof.
\end{proof}

We now extend the above model with the addition of noise arcs.  In particular,
we make the following new assumptions, which allow some receivers to receive
from senders in multiple groups, and which also allow some receivers not
to be in any group.

\begin{itemize}
\item[\bf A1']
The receivers are partitioned
as $G_0\cup G_1\cup\cdots\cup G_k$.  The $l$th influencer is adjacent
to all receivers of $G_l$, $l=1,\ldots k$.  Receivers in $G_0$ are not
adjacent to any influencer.

\item[\bf A2']  For each $G_l$, $l=1,\ldots, k$, there exists $H_l\subset G_l$
such that receivers in $H_l$ are adjacent only to senders from group $l$.
Say $|H_l|= \theta_ln_l$ (recall $n_l=|G_l|$), with $0<\theta_l\le 1$.

\item[\bf A3'] A subordinate in group $l$
is adjacent to at most $\beta_l \theta_l n_l$ receivers of $H_l$
($l=1,\ldots, k$; $0<\beta_l<1$).

\end{itemize}

For {\bf A2'} and {\bf A3'}, it is assumed that $\theta_l$ is chosen
so that $\theta_ln_l$ is integral.
Note one difference between {\bf A3} and {\bf A3'}: in {\bf A3}, we allow
for subordinates to be adjacent to all but one receiver of $G_l$, whereas in {\bf A3'}
the restriction is strengthened to at most a constant factor subset of $H_l$.

Clearly, smaller values of $\theta_l$ and larger numbers of receivers
in $G_0$  corresponds to greater amounts of noise.
The recovery theorem for this case is as follows.

\begin{theorem}
Assume {\bf A1'} to {\bf A3'}.  Let $\rho=\min_l{\theta_l}/\max_l{\theta_l}$.
For a subordinate $i$, let $z_i$ denote the number of $G_0$ nodes
adjacent to $i$.  Let $n_{\min}=\min(n_1,\ldots,n_k)$.
Provided that
\begin{equation}
\beta_l< \rho/2
\label{eq:randcond1}
\end{equation}
 for all $l=1,\ldots,k$
and
\begin{equation}
z_i\le  n_{\min}\theta_l\rho/2
\label{eq:randcond2}
\end{equation}
for all subordinates $i\in L_l$, $l=1,\ldots,k$, then
the unique solution
to \eqref{LP_basic} is given by $(\x^*,\t^*)$, where
$x^*_i=1$ if $i$ is an influencer else $x^*_i=0$,
and
$t^*_j=1$ if $j\in G_1\cup \cdots\cup G_k$ (i.e., $j$ is in an interest group)
while $t^*_j=0$
else (i.e., $j\in G_0$).
\label{thm:detnoise}
\end{theorem}
\begin{proof}
As above, the proof centers on constructing appropriate
KKT multipliers. We start with $\blambda$ and $\bmu$ defined
according to the following table.

\begin{center}
\begin{tabular}{llll}
           & $t^*_j$  & $\lambda_j$ & $\mu_j$ \\
\hline\\
$j\in H_l$   &   1    & $n_{\min}/n_l$ & $1-n_{\min}/n_l$ \\
$j\in G_l-H_l$ &  1    &   0         & 1 \\
$j\in G_0$     &  0    &  1          & 0 \\
\hline
\end{tabular}
\end{center}

We postpone defining $\bnu$ and $\xi$ until after we have verified
the first few KKT conditions.  Observe from the table that $(\t^*-A^T\x^*)_j=0$ for
$j\in G_0$ (both terms are 0) and also for $j\in H_l$ (both terms are 1),
so the KKT condition $\blambda^T(\t-A^T\x)=0$ is verified.
The conditions $\blambda+\bmu\ge\e$, $\t^T(\blambda+\bmu-\e)=0$, and
$\bmu^T(\e-\t)=0$, $\blambda\ge \bz$, $\bmu\ge\bz$  are all easily checked.

The remaining KKT conditions can be established by finding
$\bnu$
and $\xi$ so that
$(\x^*)^T(-A\blambda+\bnu+\xi\e)=0$ and
$-A\blambda+\bnu+\xi\e\ge\bz$.  Furthermore,
we require that $\bnu$ be positive in entries corresponding to influencers
and 0 in entries corresponding to subordinates.
In order for such $\bnu$ and $\xi$ to exist, it suffices to
establish that every entry of $A\blambda$ indexed by an influencer
is strictly greater than every entry of $A\blambda$ indexed by a subordinate.
If such a bound held, then there is a value, say $\omega$, such that
$A(i,:)\blambda>\omega$ for influencers $i$ while the opposite inequality
holds for subordinates $i$. Then we take $\nu_i=A(i,:)\blambda-\omega$
for $i$ an influencer, $\nu_i=0$ for $i$ a subordinate, and $\xi=\omega$
to satisfy the KKT conditions.

Observe that the value of $A(i,:)\blambda$ when $i$ is
the influencer for group $l$ is $n_{\min}/n_l\cdot|H_l|$ which
is bounded below by $\theta_ln_{\min}$.  On the other hand, when $i$ is
a subordinate in group $l$, then
$$A(i,:)\blambda \le \beta_l\theta_ln_{\min}+z_i.$$
(The first term arises from {\bf A3'}.)  Thus, to establish the KKT
conditions requires
for all $l$, all subordinates $i\in L_l$,
$$\beta_l +\frac{z_i}{n_{\min}\theta_l} <\frac{ \min_{l'}\theta_{l'}}{\theta_l}.$$
But this is established by the assumptions of the theorem, since the
two terms on the left-hand side are bounded above by $\rho/2$ (with the first bound being strict) while
the right-hand side upper bounds $\rho$.

Finally, uniqueness is established similarly as before: by complementarity, any
solution $(\x^\diamond,\t^\diamond)$ must satisfy $\bnu^T(\e-\x)=0$ for the
particular dual vector
$\bnu$ defined above. As in Theorem~\ref{thm1}, we must have $\x^\diamond = \x^*$. On the other hand,
$(\x^\diamond,\t^\diamond)$ has to satisfy $\bmu^T(\e - \t)=0$ and $\blambda^T(\t - A^T\x) = 0$
with the particular dual vectors $\bmu$ and $\blambda$ defined above. The first relation implies
that $t^\diamond_j = 1$ for all $j\in G_l - H_l$, $l=1,\ldots,k$. The second relation
implies that $t^\diamond_j = 0$ for all $j\in G_0$, and also forces $t^\diamond_j = 1$
at $j\in H_l$, $l=1,\ldots,k$.
Thus, we also have $\t^\diamond = \t^*$. This completes the proof.
\end{proof}

We now consider a randomized generative model to create a social
network in which the interest groups and influencers are `planted'
but then obscured by randomly generated noise.

Assume the sizes of the
interest groups $G_0,G_1,\ldots,G_k$ and $L_1,\ldots,L_k$ are chosen
determinimistically.  Let $q,s>0$ be two fixed parameters (independent of
problem size) of the generative model.

The arcs are chosen at random by the receivers
as follows.

\begin{enumerate}
\item
Each receiver in $G_1,\ldots,G_k$ creates an incoming arc from
its influencer with probability 1.
\item
Each receiver in $G_l$, $l=1,\ldots,k$, creates an incoming arc
with probability
$s r_{\min}/r_l$ from each subordinate in $L_l$.
Here $r_{\min}=\min(r_1,\ldots,r_k)$, and recall that
$r_l$ stands for the number of subordinates in $L_l$.

\item
With probability $q/r$, each receiver in $G_l$, $l=1,\ldots,k$,
creates an incoming arc from each sender outside its interest group.  Here,
$r=r_1+\cdots+r_k+k$, which is the total number of senders.

\item
A node in $G_0$ creates incoming arcs from each subordinates in all
groups $L_1,\ldots,L_k$ each
with probability $sr_{\min}/r$.
\end{enumerate}

One motivation for these formulas is that each receiver in $G_1,\ldots,G_k$ will have approximately
the same expected indegree, namely, approximately $1+sr_{\min}+q$, which in turn
is approximately $sr_{\min}$.
Thus, an algorithm could not distinguish
receiver interest groups with simple degree-counting.  The expected outdegree
of the influencer for group $l$ is $n_l+(n-n_l)q/r$,
and the
expected outdegree of a subordinate is $n_lsr_{\min}/r_l + |G_0|sr_{\min}/r +(n-n_l)q/r$.
Here, we set $n = n_1 + \cdots +n_k$.
This means that
an influencer can be distinguished from its own subordinates via degree counting,
but degrees alone cannot identify which $k$ senders are the influencers
(since a subordinate in $G_l$ could have higher degree than the influencer
in $G_{l'}$).  Finally, rule 4 implies that the expected indegree of nodes
in $G_0$ is roughly $sr_{\min}$, so again, they are not distinguished
by their degree.

The main theorem about this construction is that under certain
assumptions concerning the sizes of the groups,
$q$ and $s$, exact recovery of the optimal
solution is assured with high probability.

\begin{theorem}
Assume the graph is generated by rules $1$--$4$ enumerated above.
Assume
also that
\begin{equation}
s \le 0.3e^{-.4q},
\label{eq:randassump1}
\end{equation}
\begin{equation}
\ \ \ \ \ \ \ |G_0|\le 0.1n_{\min}re^{-1.3q}/(sr_{\min}),
\label{eq:randassump2}
\end{equation}
$$r\ge 6q,$$
and
$$r_l\le r/10 - 1,$$
for $l=1,\ldots,k$.

Then with probability exponentially close to $1$,
the conditions of Theorem~$\ref{thm:detnoise}$ hold, and hence the
influencers can be recovered as the solution to \eqref{LP_basic}.
By ``exponentially close to $1$'' we mean that
the probability of success is  $1-c_1\exp(-c_2n_{\min})$,
for scalars $c_1,c_2>0$ that may depend on $s$, $q$, $r_{\min}$ and $r$.
\end{theorem}

\begin{proof}
First, let us estimate how many receivers of $G_l$ will have no incoming arcs
from senders outside $G_l$ and we shall take the collection of all such receivers to be $H_l$. Moreover,
to be specific, we take $\theta_l$ so that $|H_l| = \theta_l n_l$ and set $\beta_l$ so that
a subordinate in group $l$
is adjacent to $\beta_l \theta_l n_l$ receivers of $H_l$, $l=1,\ldots,k$.

Now, note that a given receiver $j$ in $G_l$ has the probability
of  $(1-q/r)^{r-r_l-1}$ of having no out-of-group senders, which is bounded
below by $(1-q/r)^r$, which in turn is bounded below by $e^{-1.1q}$  provided
$q/r\le 1/6$ as assumed in the theorem.  On the other hand,
we have
$(1-q/r)^{r-r_l-1}\le (1-q/r)^{0.9r}\le e^{-0.9q}$
since $r_l\le 0.1r - 1$ by assumption.
Thus, the expected size of $H_l$ lies in the range $[n_le^{-1.1q},n_le^{-0.9q}]$.
The probability is thus exponentially small as $n_l$
gets large that $|H_l|$ will lie outside $[.9n_le^{-1.1q},1.1n_le^{-0.9q}]$.
Therefore, by the union bound, the probability is
exponentially small that any $|H_l|$, $l=1,\ldots,k$, will lie
outside this range. Hence, with probability exponentially
close to 1, $.9e^{-1.1q}\le \theta_l \le 1.1e^{-.9q}$ for all $l=1,\ldots,k$.
Furthermore, this means $\rho$ as defined in Theorem~\ref{thm:detnoise}
is at least $0.8e^{-0.2q}$.

Next, for each subordinate in group $l$, the probability that
a receiver in $G_l$ lies in $H_l$ and selects that particular subordinate is $sr_{\min}(1-q/r)^{r-r_l-1}/r_l$.
Thus, the expected number of receivers from $H_l$ that will select this subordinate
is $sr_{\min} (1-q/r)^{r-r_l-1}n_l/r_l$.
Hence, with probability
exponentially close to 1, the number of $H_l$ members
adjacent to this subordinate lies  in $[0.9sr_{\min}e^{-1.1q}n_l/r_l,1.1sr_{\min}e^{-.9q}n_l/r_l]$.
Thus, by the union bound, the probability is exponentially close to 1 that
all groups satisfy
\[
\beta_l\le \frac{1.1sr_{\min}e^{-.9q}}{\theta_lr_l} \le \frac{1.1sr_{\min}e^{.2q}}{.9r_l} < 1.3 se^{.2q}
< 0.4e^{-0.2q}\le \frac{\rho}{2}.
\]
Hence, we see that \eqref{eq:randcond1}
is satisfied.

Finally, we turn to the other condition of
Theorem~\ref{thm:detnoise}, i.e., \eqref{eq:randcond2}.
Observe that the expected number of
$G_0$-receivers that will select a particular subordinate is given by
$sr_{\min}|G_0|/r$. Thus, by Hoeffding's inequality,
the number of such receivers is bounded above by
\begin{equation}\label{thirdbd}
  \frac{n_{\min}}{4}(0.9 e^{-1.1q})(0.8 e^{-.2q}) + \frac{sr_{\min}|G_0|}{r} < \frac{n_{\min}}{2}(0.9 e^{-1.1q})(0.8 e^{-.2q})
\end{equation}
with probability at least
\begin{equation}\label{thirdbd2}
  1 - \exp\left(-\left[\frac{n_{\min}}{4|G_0|}(0.9 e^{-1.1q})(0.8 e^{-.2q})\right]^2|G_0|\right) \ge 1 - \exp(-c_3n_{\min})
\end{equation}
for some $c_3$ depending only on $s$, $q$ and $r_{min}/r$, where the inequalities in \eqref{thirdbd}
and \eqref{thirdbd2} follow from \eqref{eq:randassump2}. Moreover, using bounds on $\theta_l$ and $\rho$ from
the above discussions,
the right hand side of \eqref{thirdbd} is bounded above by $n_{\min}\theta_l\rho/2$ with
probability exponentially close to $1$.
Therefore, with probability exponentially close to
1, all subordinates will be adjacent to at most $n_{\min}\theta_l\rho/2$ in $|G_0|$, which establishes the theorem.
\end{proof}

\section{Probabilistic graphical model}
\label{sec:probmodel}
In this section, we consider the {\em independent cascade model}, which
was introduced by Goldenberg et al.\ \cite{Goldenberg}
and analyzed by Kempe et al.\ \cite{Kempe}.
Each arc $e\in E$ is now labeled with a
probability $p_e$.  At each time step, a node that received a message on the
previous step transmits it along an outgoing arc $e$ with probability $p_e$.
If the random choice is made not to transmit, then the sender does not
attempt to transmit again on subsequent steps.

Note that finding influential nodes in the independent cascade model
is not the same problem as finding influential nodes in a deterministic
network whose arcs have been selected probabilistically as in
the previous section.   The reason is that
in the independent cascade model, the
algorithm selecting the  most influential set of $k$ senders does not
have prior knowledge as to which transmissions will succeed or fail.

We focus again only on the bipartite graph case.  For this model, the
bipartite assumption apparently does entail a loss of generality, i.e.,
it is not clear how the general case can be reduced to the bipartite case.
On the other hand, the bipartite case still has some bearing on reality;
\cite{Bakshy} shows that the most common cascade depth on the Twitter social
network is 1.

Thus, assume $G=(V_1,V_2,E)$ is a bipartite graph.
Based on this fixed $G$, we consider a family of graphs $\Upsilon$ generated from
$G$ having the same vertex sets but with arcs
chosen from $E$ with independent probability $p_e$.

This model
can then be formulated as the following stochastic integer programming problem:
  \begin{equation}\label{sto_opt}
    \begin{array}{rl}
      \max\limits_{\x,\t} & E_{A\in \Upsilon}(\e^T\t)\\
      {\rm s.t.} & \t \le A^T\x,\\
                 & \bz\le \t\le \e, \bz\le \x\le \e,\\
                 & \e^T \x = k, x_i \in \{0,1\}, i = 1,\ldots,|V_1|,
    \end{array}
  \end{equation}
where the expectation is taken over the $|V_1|\times |V_2|$ incidence matrices $A$ of graphs in
$\Upsilon$. Notice that for any zero-one vector $\x$ satisfying $\e^T\x = k$, the
corresponding feasible random variables $t_j$, $j\le \sum_{l}n_l$, for \eqref{sto_opt}
that maximize the expectation satisfy
  \begin{equation*}
    t_j(A) = \begin{cases}
      1 & {\rm \ if\ } \exists i {\rm \ s.t.\ arc \ } (i,j) {\rm \ is \ chosen \ in\ } A,\\
      0 & {\rm \ otherwise.}
    \end{cases}
  \end{equation*}
  Furthermore, the probability that there is no arc $(i,j)$ for a given $j$ is
$\prod_{i:(i,j)\in E}(1-p_{(i,j)})^{x_i}.$
 Hence,
  \[
  E_{A\in \Upsilon}(t_j) = 1 - \prod_{i:(i,j)\in E}(1-p_{(i,j)})^{x_i},
  \]
  from which we see immediately that problem \eqref{sto_opt} is the same as
  \begin{equation*}
    \begin{array}{rl}
      \min\limits_\x & \sum_{j\in V_2} \prod_{i:(i,j)\in E}(1-p_{(i,j)})^{x_i} \\
      {\rm s.t.} & \e^T \x = k, \bz\le \x\le \e, x_i \in \{0,1\}, i = 1,\ldots,|V_1|.
    \end{array}
  \end{equation*}
  Dropping the integer constraints, we obtain the following
relaxation to \eqref{sto_opt}:
  \begin{equation}\label{relax}
    \begin{array}{rl}
      \min\limits_\x & g(\x):=\sum_{j\in V_2} \prod_{i:(i,j)\in E}(1-p_{(i,j)})^{x_i} \\
      {\rm s.t.} & \e^T \x = k, \bz\le \x\le \e.
    \end{array}
  \end{equation}
It can be checked that the objective function denoted $g(\x)$ is a convex
function of $\x$ for a fixed probability vector.
Thus, the above is a
convex optimization problem and can be solved in polynomial time.

We now begin the analysis of the possibility that the solution to the
stochastic integer program can be recovered from the convex relaxation.
Again, we assume a partitioning of both senders and receivers into $k$
interest groups with one influencer
per interest group.
We also assume that an influencer is adjacent to all receivers in its corresponding group,
and that there exists a collection of receivers $G_0$ with $|G_0|\ge 0$
that are not adjacent to any influencer. In other words, we assume {\bf A1'}.

First, it should be noted that even in the presence of strong assumptions
{\bf A1} to {\bf A3} made in the
deterministic case, the convex relaxation is not guaranteed to
find the influencers.
In
fact, even the integer solution may not
find the influencers.  For example, consider the case in which there
are two interest groups ($k=2$), and each has an influencer and one subordinate, thus
four senders total.  Assume the receiver group sizes are $n_1=100$ and
$n_2=20$.  Suppose that the number of receivers connected to the
two subordinates are $m_1=99$ and $m_2=10$ respectively.
Finally, suppose all the arc
probabilities are 0.5.  In this case, the optimal integer solution is to take
both the influencer
and subordinate in the first group rather than the two influencers.  This is because
the influencer in the large group will reach only about 50 of its receivers, so
its subordinate will reach another 25 or so in the first
group, which is better than the 10 or so that the influencer of the second group
might reach.

This example indicates that the influencers are in fact not the most influential
nodes unless the group sizes are not too disparate.

Now consider again a similar example in which the
sizes are $n_1=100$, $n_2= 44$, $m_1=80$, $m_2=40$.  In this case,
one can check that the optimal integer solution picks out the two
influencers and reaches an expected $0.5\cdot 144=72$ receivers.  However,
it is not hard to check that
there is a continuous solution of the
form $x_1=1$, $x_2=\epsilon$, $x_3=1-\epsilon$, $x_4=0$ for an $\epsilon>0$
better than this integer solution.

These small examples indicate that
two extensions to the
analysis from the last section should be made to handle the cascade model.
First, the smaller interest groups cannot be too much smaller than the larger
ones, else they will never be selected even by the
integer programming model.  Second, even when the
convex relaxation succeeds, it often gives weights to influencers that are close to
1 but not equal to 1.  In other words, a rounding procedure must be established
to obtain the integer solution from the convex solution.

To continue the development of the model, let us simplify notation
by assuming that all arcs have exactly the same probability $p\in(0,1)$.  (It is likely that our
results can be extended to the general case of distinct $p_e$ values,
but there is no obvious a priori
model for selecting values of $p_e$ that would be more realistic than equal
values.)
This means that the objective function may be rewritten as
$$g(\x)=\sum_{j\in V_2}(1-p)^{\h_j^T\x},$$
where $H$ is the $|V_1|\times |V_2|$ matrix whose $(i,j)$th entry is $1$ if
there is an arc in $G$ from the $i$th node in $V_1$ to
the $j$th node in $V_2$, else $H$ is zero, and where $\h_j$ denotes the
$j$th column of $H$.

We show in the next theorem that, under some assumptions and using a
suitable rounding procedure, it is possible to identify the influencers
from a solution of problem \eqref{relax}. Moreover, the indicator vector $\x^*$ corresponding to
the influencers actually solves \eqref{sto_opt}.

\begin{theorem}\label{thm6}
Suppose that for some $\xi\in \left[0,\frac{1}{2k+1}\right)$, we have
\begin{equation}\label{probmodelcond}
\min_{1\le i\le k} \hat n_i \ge (1-p)^{0.5 + \frac{\xi}{2}}\max_{1\le j\le k} \left\{\alpha_j + \frac{\gamma_j}{1-p}\right\}\ {\rm and}\
    n_l-\alpha_l > (1-p)^{-k}\gamma_l, \ \forall l=1,\ldots,k,
    \end{equation}
where $\hat n_l$ denotes the number of receivers in $G_l$ that
are not adjacent to senders outside $L_l$; each influencer in $L_l$ is adjacent
to all receivers in $G_l$;
each subordinate in $L_l$
is adjacent to at most $\alpha_l < n_l$ receivers
in $G_l$ and at most $\gamma_l$
receivers outside $G_l$.

Define a vector $\y_\xi(\x)$ as follows:
  \begin{equation*}
    (\y_\xi(\x))_i := \begin{cases}
      1& {\rm if}\ x_i \ge 0.5 - \frac{\xi}{2},\\
      0& {\rm otherwise}.
    \end{cases}
  \end{equation*}
    Then:
    \begin{enumerate}[{\rm (a)}]
      \item The vector $\y_\xi(\x^\diamond) = \x^*$ for any solution $\x^\diamond$ of problem \eqref{relax}.
      \item The vector $\x^*$ is a solution of \eqref{sto_opt}.
    \end{enumerate}
  \end{theorem}
  \begin{remark}
    When there are no noise arcs, i.e., $\gamma_l =0$ for all $l$,
    taking $\xi = 0$ and noting that $\hat n_l = n_l$ in this case,
    the condition~\eqref{probmodelcond} reduces to
    $\min_i n_i \ge (1-p)^{0.5}\max_j \alpha_j$.  This is the probabilistic
    noiseless case, that is, the analog to the deterministic noiseless case given
    by {\bf A1}--{\bf A3} in the previous section.  If $p\rightarrow 1$ (i.e.,
    the deterministic limit is approached), then
    the restriction $\min_i n_i \ge (1-p)^{0.5}\max_j \alpha_j$ becomes arbitrarily
    loose.
  \end{remark}
  \begin{proof}
    Let $\x^\diamond$ be a solution of \eqref{relax}. We first analyze the case when
$x^\diamond_i \ge 1 - \xi > \frac{2k}{2k+1}$ at every influencer $i$.
    In this case, it holds that for any $x_j^\diamond$ with $j$ being a subordinate,
    \[
    x_j^\diamond \le k - \sum_{i:{\rm influencer}}x_i^\diamond < k - \frac{2k^2}{2k+1} = \frac{k}{2k+1} < \frac{1}{2} - \frac{\xi}{2},
    \]
    from which we see immediately that $\y_\xi(\x^\diamond) = \x^*$.

    Hence, to establish {\rm (a)}, it remains to analyze the case when there exists a sender group $L_{l_0}$ such that $x^\diamond_{i_0} <  1-\xi$ at the influencer $i_0\in L_{l_0}$.

    In this case, first, we claim that $x^\diamond_j = 0$ for all $j\in L_{l_0}$, $j\neq i_0$.

    Suppose
    to the contrary that $x^\diamond_{j_0} > 0$ for some such $j = j_0$. We shall
    establish that $(\nabla g(\x^\diamond))_{i_0} < (\nabla g(\x^\diamond))_{j_0}$. Granting this,
    one can readily show that the vector $x^\dagger_\epsilon$
    \[
   (x^\dagger_\epsilon)_i = \begin{cases}
      x^\diamond_{i_0} + \epsilon x^\diamond_{j_0} & {\rm if}\ i = i_0,\\
      (1-\epsilon)x^\diamond_{j_0} & {\rm if}\ i = j_0,\\
      x^\diamond_{i} & {\rm otherwise}.
    \end{cases}
    \]
    is feasible for \eqref{relax} and $g(\x^\dagger_\epsilon)< g(\x^\diamond)$ for all sufficiently small $\epsilon>0$, contradicting the optimality of $x^\diamond$.
    Hence, to establish the claim, it now remains to show that $(\nabla g(\x^\diamond))_{i_0} < (\nabla g(\x^\diamond))_{j_0}$.

    To this end, define the sets $V_{i_0} = \{j:\; (i_0,j)\in E\}$ and $V_{j_0} = \{j:\; (j_0,j)\in E\}$. Suppose first that $\gamma_{l_0} = 0$.
    Since $n_{l_0} > \alpha_{l_0}$ by assumption, it follows that $V_{j_0}\subsetneq V_{i_0}$ and hence
    we see immediately that
    \[
     (\nabla g(\x^\diamond))_{i_0} = \ln(1-p) \sum_{j\in V_{i_0}}(1-p)^{\h_j^T\x^\diamond} < \ln(1-p) \sum_{j\in V_{j_0}}(1-p)^{\h_j^T\x^\diamond} = (\nabla g(\x^\diamond))_{j_0}.
    \]
    We next consider the case when $\gamma_{l_0} > 0$:
    \begin{equation*}
    \begin{split}
      (\nabla g(\x^\diamond))_{i_0} &= \ln(1-p) \sum_{j\in V_{i_0}\cap V_{j_0}}(1-p)^{\h_j^T\x^\diamond} + \ln(1-p) \sum_{j\in V_{i_0}\backslash V_{j_0}}(1-p)^{\h_j^T\x^\diamond}\\
      &\le \ln(1-p) \sum_{j\in V_{i_0}\cap V_{j_0}}(1-p)^{\h_j^T\x^\diamond} + \ln(1-p) (n_{l_0}-\alpha_{l_0})
(1-p)^k\\
      &< \ln(1-p) \sum_{j\in V_{i_0}\cap V_{j_0}}(1-p)^{\h_j^T\x^\diamond} + \ln(1-p) \gamma_{l_0}\\
      &\le \ln(1-p) \sum_{j\in V_{i_0}\cap V_{j_0}}(1-p)^{\h_j^T\x^\diamond} + \ln(1-p) \sum_{j\in V_{j_0}\backslash V_{i_0}}(1-p)^{\h_j^T\x^\diamond}\\
      &=  (\nabla g(\x^\diamond))_{j_0},
    \end{split}
    \end{equation*}
    where the first inequality follows from the fact that there are at least
$n_{l_0}-\alpha_{l_0}$ nodes in $V_{i_0}\backslash V_{j_0}$, and $\h_j^T\x^{\diamond}\le k$ by feasibility
(since $\e^T\x^{\diamond}=k$ and $\x^{\diamond}\ge \bz$).
    The second inequality follows from the assumption of the theorem, while the third inequality follows from the definition of $\gamma_{l_0}$
    and the fact that $\h_j^T\x^\diamond\ge 0$. Combining the two cases, we conclude that $(\nabla g(\x^\diamond))_{i_0} < (\nabla g(\x^\diamond))_{j_0}$ and hence
    we have shown that $x^\diamond_j = 0$ for all $j\in L_{l_0}$, $j\neq i_0$.

    Using this claim, the fact that $x^\diamond_{i_0}<1 - \xi$, $0\le \x^\diamond \le 1$ and $\e^T\x^\diamond = k$,
    we conclude that there must exist a group $L_{l_1}$ such that $x^\diamond_{i_1}\ge 1 - \xi$ at
    the influencer $i_1\in L_{l_1}$ with $x^\diamond_{j_1} > 0$ for some subordinate $j_1\in L_{l_1}$. Define
    \begin{equation*}
    \begin{array}{ll}
      a := \min\{x^\diamond_i:\; \mbox{$i$ is an influencer}, x^\diamond_i<1 - \xi\},&i_a\in \argmin\{x^\diamond_i:\; \mbox{$i$ is an influencer}, x^\diamond_i<1-\xi\},\\
      b := \max\{x^\diamond_j:\; \mbox{$j$ is a subordinate}, x^\diamond_j>0\},&j_b\in \argmax\{x^\diamond_j:\; \mbox{$j$ is a
      subordinate}, x^\diamond_j>0\}.
    \end{array}
    \end{equation*}
    From the above discussion, these quantities are well-defined. For easy reference, we name the group containing $i_a$ by $L_{l_a}$
    and the group containing $j_b$ by $L_{l_b}$. We also denote the influencer in $L_{l_b}$ by $i_b$.

    To establish that $\y_\xi(\x^\diamond) = \x^*$,
we now show that  $b<0.5 - \frac{\xi}{2}\le a$.

    To this end, recall that
    the point $\x^\diamond$ is optimal if and only if there exist $\u\ge \bz$, $\v\ge \bz$ and $\lambda\in \R$  such
    that the following KKT conditions are satisfied:
    \begin{align}\label{nonlinopt}
    \begin{array}{rrr}
       \nabla g(\x^\diamond)+ \lambda \e - \u + \v  = \bz,\\
       \u^T\x^\diamond  = 0,\ \ \v^T(\x^\diamond - \e)  = 0.
    \end{array}
    \end{align}
    Since $x^\diamond_{i_a} = a <1 - \xi \le 1$, we have $v_{i_a} = 0$. Moreover, since there are $\hat n_{l_a}$
    receivers adjacent to no vertex outside $L_{l_a}$ and recall we have shown that $x_i^\diamond = 0$
    for $i\in L_{l_a}$, $i\neq i_a$, it follows that $\h_j^T\x^\diamond = a$ at all such receivers.
    Combining these two observations, we obtain that
    \begin{equation}\label{noisybd1}
      \ln(1-p)\cdot \hat n_{l_a}(1-p)^a \ge (\nabla g(\x^\diamond))_{i_a} \ge -\lambda.
    \end{equation}
    Next, notice that at $j_b$, we have $x^\diamond_{j_b} = b >0$ and thus
    $u_{j_b}=0$. Hence
    \begin{align*}
      -\lambda &\ge (\nabla g(\x^\diamond))_{j_b} = \left(\ln(1-p)H\begin{pmatrix}
        (1-p)^{\h_j^T\x^\diamond}
      \end{pmatrix}_{j\le \sum_{l}n_l}\right)_{j_b}\\
      & = \ln(1-p)\sum_{j: (j_b,j)\in E} (1-p)^{\h_j^T\x^\diamond}\\
      & \ge \ln(1-p)\cdot (\alpha_{l_b}(1-p)^{1+b} + \gamma_{l_b}(1-p)^b),
    \end{align*}
    where the inequality follows since:
    at any $j\in G_{l_b}$ with $(j_b,j) \in E$, $\h_j$ is $1$ at the $i_b$th and $j_b$th entry,
    and hence $\h_j^T\x^\diamond \ge \x^\diamond_{i_b} + \x^\diamond_{j_b}\ge 1+b$;
    while for those $j\notin G_{l_b}$ with $(j_b,j) \in E$, we have $\h_j^T\x^\diamond \ge x^\diamond_{j_b}\ge b$. Combining this with \eqref{noisybd1}, we obtain further that
    \begin{align*}
    \begin{array}{rrl}
      &\ln(1-p)\cdot \hat n_{l_a}(1-p)^a & \!\!\!\!\ge  \ln(1-p)\cdot (\alpha_{l_b}(1-p)^{1+b} + \gamma_{l_b}(1-p)^b),\\
     \Rightarrow& \alpha_{l_b}(1-p)^{1+b} + \gamma_{l_b}(1-p)^b & \!\!\!\!\ge \hat n_{l_a}(1-p)^a,\\
      \Rightarrow&(1-p)^{b}& \!\!\!\!\ge\frac{\hat n_{l_a}}{\alpha_{l_b} + (1-p)^{-1}\gamma_{l_b}}(1-p)^{a-1}\ge (1-p)^{a - 0.5 + \frac{\xi}{2}},
    \end{array}
    \end{align*}
    where the last inequality comes from the assumption. This implies that $b \le a - 0.5 + \frac{\xi}{2}$, which
    together with $b\ge 0$ and $a<1 - \xi$ gives what we want. This proves part {\rm (a)}.

    We now prove part {\rm (b)}.

    Take a feasible point $\x\neq \x^*$ of \eqref{sto_opt}. Necessarily, there are exactly $k$ entries of $\x$ equal to $1$.
    We will establish {\rm (b)} by constructing a feasible vector $\x'$ from $\x$ such that $g(\x')<g(\x)$.

  Consider first
  the case that there is a group $L_{l_0}$ with at least one
subordinate $j_0$ such
that $x_{j_0}=1$, while for the influencer of the group $i_0$, $x_{i_0}=0$.
In this case, define
  a feasible vector $\x'$ by
  \begin{equation*}
    x'_i = \begin{cases}
      1 & {\rm if \ }i = i_0,\\
      0 & {\rm if \ }i = j_0,\\
      x_i & {\rm else},
    \end{cases}
  \end{equation*}
  and note that
  \[
  g(\x) = \sum_{j\in G_{l_0}}(1-p)^{\h_j^T\x} + \sum_{l\neq l_0}\sum_{j\in G_{l}}(1-p)^{\h_j^T\x}.
  \]
  We shall analyze the change in the value of $g$ by looking at contributions from within $G_{l_0}$ and outside $G_{l_0}$.

  By changing from $\x$ to $\x'$, there are now at least $n_{l_0}-\alpha_{l_0}$ receivers within group $G_{l_0}$ adjacent to one
  more sender (the influencer $i_0$). Since in solution
  $\x$ these receivers were adjacent to at most $k-1$ senders, this means that
  the objective function contribution from $G_{l_0}$ goes down by at least
  $(n_{l_0}-\alpha_{l_0})((1-p)^{k-1}-(1-p)^k) = (n_{l_0}-\alpha_{l_0})p(1-p)^{k-1}$.
  On the other hand, the objective function may increase due to contributions in other groups;
  in particular, the subordinate $j_0$ may
  be adjacent to at most $\gamma_{l_0}$ receivers in other groups, and therefore the increase
  in the objective function is at most $\gamma_{l_0}((1-p)^0-(1-p)^1) = p\gamma_{l_0}$. Thus, to confirm that $g(\x') < g(\x)$ requires showing
  that $(n_{l_0}-\alpha_{l_0})p(1-p)^{k-1}>p\gamma_{l_0}$; this follows from the second
  half of \eqref{probmodelcond}.

The preceding
argument shows that a solution to \eqref{sto_opt} cannot be optimal if a
subordinate in an interest group is selected while the influencer is not.
In particular, this means that if a solution $\x$ is optimal and it
has exactly one `1' entry per influence group, then it must be equal to
$\x^*$.

Consider now the case that
  $\x$ has two (or more) `1' entries in the same group $L_{l_1}$.  By
the preceding analysis,
we already know that $\x$ is suboptimal if the influencer
in $L_{l_1}$ is not selected.  Therefore, assume that $x_{i_1}=1$, where
$i_1$ is the influencer of $L_1$, and assume also that there is a subordinate
$j_1\in L_{l_1}$ such that $x_{j_1}=1$.

By feasibility, there is
  another group $L_{l_0}$ in which $\x$ has no `1' entry at all. Consider the solution
  $\x'$ in which the subordinate in group $L_{l_1}$ indexed $j_1$ is changed to 0,
  while the influencer in group $L_{l_0}$, say which is numbered $i_0$, is changed to 1.

  Unlike the previous case, we shall analyze the change in the value of $g$ by looking at the decrease of function value
  due to the change from $x_{i_0}=0$ to $x_{i_0}'=1$, and then the increase induced by changing $x_{j_1}=1$ to $x_{j_1}'=0$.

  Since there are $\hat n_{l_0}$ receivers in $G_{l_0}$ not adjacent to any sender whose $\x$-value is 1,
  the decrease in the objective function due to changing $x_{i_0}=0$ into $x_{i_0}'=1$
  is at least $\hat n_{l_0}((1-p)^0-(1-p)^1) = \hat n_{l_0}p$.  On the other hand, the increase
  in the objective function due to changing $x_{j_1}=1$ to $x_{j_1}'=0$ is at most
  \[
  \alpha_{l_1}((1-p)^1-(1-p)^2)+\gamma_{l_1}((1-p)^0-(1-p)^1) = \alpha_{l_1}p(1-p)+\gamma_{l_1}p,
  \]
  where: the first term accounts for
  the maximum possible increase in the objective function among receivers in
  $G_{l_1}$ (these receivers are adjacent to at least two
  senders in $L_{l_1}$ in solution $\x$, namely the $i_1$ and $j_1$), while the second term is the maximum possible increase contributed by other groups (not $G_{l_1}$)
  due to changing $x_{j_1}=1$ to $x_{j_1}'=0$. Thus, showing that the objective function decreases
  requires establishing  $\hat n_{l_0}p>\alpha_{l_1}p(1-p)+\gamma_{l_1}p$.  This follows
  from the first condition of \eqref{probmodelcond} since $1-p < (1-p)^{0.5+\xi/2}$. This completes the proof.
  \end{proof}

  It is now possible, as in the previous section, to write down rules for
a generative model whose networks will satisfy the conditions of
Theorem~\ref{thm6}.  Since the construction and proof are not
very different from those in the previous section, we will omit the
details but instead point out the salient differences imposed
by \eqref{probmodelcond}.  The first part of condition \eqref{probmodelcond}
requires all the receiver groups $G_1,\ldots,G_k$ to have roughly the
same size.  The second part of the condition places
a stringent bound on the number of noise arcs if $k$ is
large. We conjecture that a different analysis could improve this
exponential dependence on $k$.

\section{On solving the general case with the convex relaxation}
The theory developed shows that the two convex relaxations
can exactly solve the underlying integer problem when the data comes
from the postulated generative models.  Unless $P=NP$, we cannot
expect our convex relaxation (or any convex relaxation) to solve
these problems if the data comes from an unknown source.  It is reasonable,
however, to at least expect that when the relaxations succeed in
exact recovery for general data, there is a certificate of their
success.

We first make the fairly obvious but still useful observation that
if the LP model \eqref{LP_basic} returns a 0-1 solution as the
LP optimizer, then this
solution must be optimal also for the integer program, and furthermore,
optimality for the integer program is certified by the LP solution.
This observation holds regardless of the source of the data.
We use this fact in the next section.

In the case of the convex relaxation \eqref{relax} for the cascade
model, the situation is not as clear. Our theory states that
even when \eqref{relax} is able to identify the optimizer of
\eqref{sto_opt}, it does not
return a 0-1 solution
and hence is not able to certify optimality.  For general problems,
the proposed rounding
procedure may not even yield a feasible point.
Thus, in the case that the problem data comes from an unknown source,
it is unclear how the convex solution could be useful.

We now describe a simple strategy for making use of the convex solution
of \eqref{relax}.
Consider
  \begin{equation*}
    \tilde{\y}(\x) := \begin{cases}
      1 & \mbox{if $x_i$ is one of the $k$ largest entries in $\x$},\\
      0 & \mbox{else}.
    \end{cases}
  \end{equation*}
  Notice that this vector is well-defined whenever the $k$ largest entries in $\x$ are uniquely identified.
  In addition, $\tilde{\y}(\x^\diamond) = \y_\xi(\x^\diamond)$
  under the assumptions of Theorem~\ref{thm6}. Furthermore, it is not hard to show that $\x \mapsto \tilde{\y}(\x)$
  sends an $\x$ feasible for \eqref{relax} to the closest vertex of the feasible region. Given a solution $\x'$ from the
  convex relaxation \eqref{relax}, let $\x''$ be a solution of the following convex optimization problem
  \begin{equation*}
    \begin{array}{rl}
      \min\limits_\x & g(\x)\\
      {\rm s.t.} & \e^T \x = k, \\
                 & \tilde{\y}(\x')^T\x \le k - 1,  \\
                 &\bz\le \x\le \e.
    \end{array}
  \end{equation*}
  It is not hard to see that the constraint $\tilde{\y}(\x')^T\x \le k - 1$ cuts off one and only one vertex from the feasible region of \eqref{relax},
  namely $\tilde{\y}(\x')$. Thus, it follows immediately that if $g(\x'') > g(\tilde{\y}(\x'))$, then one can certify that $\tilde{\y}(\x')$ is
  an optimal solution for \eqref{sto_opt}.  More complicated variants of this
strategy exist that require greater computational time but are able to
certify optimality of $\tilde{\y}(\x')$ in more cases.  We have confirmed that the
simple strategy described in this section is able to certify
optimality to \eqref{sto_opt} for some instances in which $\x'$ is already
fairly close to a 0-1 point.

\section{A simplified forest fire model with numerical simulations}
\label{sec:forestfire}

  In this section, via numerical simulations, we study the performance
  of \eqref{LP_basic} and \eqref{relax} on recovering the influencers
  in random graphs generated according to a simplified forest fire
  model.  Our codes are written in MATLAB. All numerical experiments
  are performed on MATLAB 7.14 (R2012a) equipped with CVX version 1.22 \cite{GrantBoyd}
  and SeDuMi 1.21 \cite{sedumi}.

  We generate random graphs as follows. We start with $k$ influencers,
  each paired up with one receiver, and set upper bounds $u_i$ and
  $u_f$ for the total number of senders and
  receivers, respectively. When the upper bounds $u_i$ and $u_f$ are
  not reached, we add a receiver with probability $p_1$, and a
  subordinate with probability $1 - p_1$. The new receiver $j$ first
  picks randomly an existing receiver and chooses one of its
  senders $i_1$ as its own at random, i.e., an arc
  $(i_1,j)$ is added to the graph.  Then, with probability
  $p_2$, this new receiver $j$ continues by picking a random receiver
  $j_1$ of $i_1$, and chooses at random one of its senders $i_2$
  as its own. This process continues with probability $p_2$. The
  procedure for adding a new subordinate is similar. When one of the
  upper bounds $u_i$ and $u_f$ is reached, say $u_i$ is reached, a new
  receiver is then added according to the above procedure until $u_f$
  is also reached. This generative process is a two-layer version
  of the {\em forest-fire} model due to
\cite{Leskovec}.  Networks
  generated by the forest-fire model have graph properties that seem
  to match those of real social networks, although a detailed analysis
  of these properties is lacking.  Notice the ``rich-get-richer'' flavor
  of the forest-fire model, i.e., that nodes with many connections attract
  even more connections compared to isolated nodes; this also appears
  to be a characteristic of real social networks.

  We have further tweaked the model as follows.
  In order to guarantee that the $k$
  influencers remain most influential in the resulting graph, each
  receiver will randomly pick one influencer and add the
  corresponding arc if it is not
  already adjacent to one. Finally, a fixed percentage $\sigma\%$ of arcs
  randomly chosen from the complement graph are added as noise arcs.

  In our first test below, we consider model \eqref{LP_basic}. We
  choose $k$ between $20$ and $120$, $u_i = 10k$ and $u_f = 10u_i$. We
  consider $p_1 = 0.3$, $0.7$, $p_2 = 0.9$ and $\sigma = 0.5$, $1$.  We
  generate $10$ random instances as above using these parameters. The
  computational results, averaged over the $10$ instances, are
  reported in Table~\ref{t1}, where we report the number of arcs
  before noise is added ({\bf E$_{\rm orig}$}) and the number of noise
  arcs ({\bf E$_{\rm noise}$}). We also report the recovery error
  ({\bf err}) given by $\sqrt{\sum_{i=1}^k |x_i - 1|^2}$,\footnote{Note that by construction,
  the influencers are located at the first $k$ entries.} where $\x$ is
  the approximate solution of \eqref{LP_basic} obtained via CVX
  (calling SeDuMi 1.21), and the number of instances with successful
  recovery (${\bf N}_{\rm rec}$) marked by $\sqrt{\sum_{i=1}^k |x_i -
    1|^2} < 10^{-8}$.  The results show that even with a relatively
  large number of noise arcs, model \eqref{LP_basic} still
  successfully identify the influencers.

  In our second test, we consider model \eqref{relax}. We take $p_1 = 0.3$, $0.7$ and $p_2 = 0.9$ as before
  but consider the much smaller noise $\sigma = 0$ and $0.01$.
  Moreover, since \eqref{relax} is solved by CVX (calling SeDuMi 1.21) via a successive approximation method that becomes very costly for large instances, we
  only consider values of $k$ between $20$ and $45$.
  This is because \eqref{relax} involves a transcendental convex function
  and therefore is not expressible as a semidefinite programming problem;
  it is only approximately expressible \cite{GrantBoyd}.
  We then set $u_i = 10k$ and $u_f = 10u_i$ as before and take $p = 0.9$ in \eqref{relax}.
  We generate $10$ random instances using these parameters. The computational results averaged over the $10$ instances
  are reported in Table~\ref{t2}, where {\bf E$_{\rm orig}$} and {\bf E$_{\rm noise}$} are defined as above. The recovery error
  {\bf err} is given by $\sqrt{\sum_{i=1}^k |\tilde x_i - 1|^2}$, where
  $\tilde{\x}$ is the zero-one vector that is one at those entries corresponding to the largest $k$ elements in the solution vector returned from CVX.
  Furthermore, successful recovery is marked by $\sqrt{\sum_{i=1}^k |\tilde x_i - 1|^2} < 10^{-8}$, and the number
  of such instances is reported under ${\bf N}_{\rm rec}$.
  The computational results show that model \eqref{relax} is not capable of identifying all influencers correctly, even when there are no noise arcs.

  \begin{table}[h!]
\caption{\small Results on model \eqref{LP_basic} applied to simplified forest fire model.}
 \label{t1}\centering
\begin{tabular}{|c c||r r r r|r r r r|}
\hline
& & \multicolumn{4}{c|}{$\sigma = 0.5$}
& \multicolumn{4}{c|}{$\sigma = 1$}\\
$k$ & $p_1$ & \multicolumn{1}{c}{\bf E$_{\rm orig}$} &
\multicolumn{1}{c}{\bf E$_{\rm noise}$} & \multicolumn{1}{c}{\bf err}& \multicolumn{1}{c|}{${\bf N}_{\rm rec}$}
& \multicolumn{1}{c}{\bf E$_{\rm orig}$} &
\multicolumn{1}{c}{\bf E$_{\rm noise}$} & \multicolumn{1}{c}{\bf err}& \multicolumn{1}{c|}{${\bf N}_{\rm rec}$}\\

\hline

 20  & 0.3 &   9338 &   1953 & 0.0e+0 & 10/10   &   9367 &   3906 & 0.0e+0 & 10/10\\
 20  & 0.7 &   8674 &   1957 & 0.0e+0 & 10/10   &   8467 &   3915 & 0.0e+0 & 10/10\\
 40  & 0.3 &  18636 &   7907 & 0.0e+0 & 10/10   &  18494 &  15815 & 0.0e+0 & 10/10\\
 40  & 0.7 &  16704 &   7916 & 0.0e+0 & 10/10   &  17657 &  15823 & 0.0e+0 & 10/10\\
 60  & 0.3 &  27358 &  17863 & 0.0e+0 & 10/10   &  26842 &  35732 & 0.0e+0 & 10/10\\
 60  & 0.7 &  25548 &  17872 & 0.0e+0 & 10/10   &  26791 &  35732 & 0.0e+0 & 10/10\\
 80  & 0.3 &  35618 &  31822 & 0.0e+0 & 10/10   &  36987 &  63630 & 0.0e+0 & 10/10\\
 80  & 0.7 &  33603 &  31832 & 0.0e+0 & 10/10   &  33338 &  63667 & 0.0e+0 & 10/10\\
 100 & 0.3 &  44394 &  49778 & 0.0e+0 & 10/10   &  43628 &  99564 & 5.1e+0 &  0/10\\
 100 & 0.7 &  41670 &  49792 & 0.0e+0 & 10/10   &  43082 &  99569 & 4.9e+0 &  0/10\\
 120 & 0.3 &  54052 &  71730 & 0.0e+0 & 10/10   &  53695 & 143463 & 6.6e+0 &  0/10\\
 120 & 0.7 &  52145 &  71739 & 0.0e+0 & 10/10   &  52268 & 143478 & 6.5e+0 &  0/10\\
\hline
\end{tabular}
\end{table}

  \begin{table}[h!]
\caption{\small Results on model \eqref{relax} applied to simplified forest fire model, with $p = 0.9$.}
 \label{t2}\centering
\begin{tabular}{|c c||r r r|r r r r|}
\hline
& & \multicolumn{3}{c|}{$\sigma = 0$}
& \multicolumn{4}{c|}{$\sigma = 0.01$}\\
$k$ & $p_1$ & \multicolumn{1}{c}{\bf E$_{\rm orig}$} &\multicolumn{1}{c}{\bf err}& \multicolumn{1}{c|}{${\bf N}_{\rm rec}$}
& \multicolumn{1}{c}{\bf E$_{\rm orig}$} &
\multicolumn{1}{c}{\bf E$_{\rm noise}$} &\multicolumn{1}{c}{\bf err}& \multicolumn{1}{c|}{${\bf N}_{\rm rec}$}\\

\hline

 20 & 0.3 &   9183 & 1.0e-1 &  9/10 &   9105 &     39 & 1.0e-1 &  9/10 \\
 20 & 0.7 &   8652 & 4.4e-1 &  6/10 &   8797 &     39 & 4.0e-1 &  6/10 \\
 25 & 0.3 &  11622 & 3.0e-1 &  7/10 &  11495 &     61 & 1.0e-1 &  9/10 \\
 25 & 0.7 &  10645 & 5.4e-1 &  5/10 &  11160 &     61 & 3.0e-1 &  7/10 \\
 30 & 0.3 &  14183 & 2.0e-1 &  8/10 &  14254 &     89 & 2.0e-1 &  8/10 \\
 30 & 0.7 &  12889 & 7.0e-1 &  3/10 &  13063 &     89 & 5.4e-1 &  5/10 \\
 35 & 0.3 &  15602 & 2.0e-1 &  8/10 &  16814 &    121 & 1.0e-1 &  9/10 \\
 35 & 0.7 &  14162 & 8.4e-1 &  2/10 &  15858 &    121 & 4.4e-1 &  6/10 \\
 40 & 0.3 &  18474 & 5.4e-1 &  5/10 &  18343 &    158 & 1.0e-1 &  9/10 \\
 40 & 0.7 &  17592 & 6.4e-1 &  4/10 &  17020 &    158 & 1.0e+0 &  1/10 \\
 45 & 0.3 &  20263 & 3.0e-1 &  7/10 &  21037 &    200 & 3.0e-1 &  7/10 \\
 45 & 0.7 &  19714 & 7.4e-1 &  3/10 &  19266 &    201 & 1.1e+0 &  1/10 \\
\hline
\end{tabular}
\end{table}

\section{Conclusions}

We have considered the possibility of using convex relaxation to solve the
NP-hard problem of finding the set of $k$ most influential nodes in a social
network.  We restricted attention to the bipartite case, which is without
loss of generality when the arcs are deterministic.
We describe a generative model in which senders and receivers are both
divided into interest groups, each interest group has one influential sender,
and most of the arcs join senders in an interest group to receivers in the
same group.
Our theory
shows that for deterministic arcs, recovery of the influencers is possible
even with substantial noise. Recovery in the probabilistic model is also
possible with more stringent assumptions.  Our computational tests on the
forest-fire model, which is not covered by our theory, nonetheless exhibit
the results predicted by the theory.

The first question left by our work
is whether a stronger convex relaxation is possible
in the case of the probabilistic graph model.  S.~Ahmed pointed out in private
communication that while the problem is still in integer form, there are many
possible adjustments that could be made to the objective function before
passing to the convex relaxation; the adjustments could be chosen so that
the integer problem is not affected but the convex relaxation is stronger.

The second main question left by our work is whether a theoretical analysis
of the forest-fire model is possible.  As mentioned in the introduction,
this model is believed to correspond to real social networks much better
than the interest-group model developed herein.

The last question is whether the analysis can be extended to the nonbipartite
directed graph case.  An immediate difficulty with this case, assuming the
independent cascade model, is that there is apparently no closed-form expression
for the objective function (expected number of receivers reached by the $k$
senders) for the optimization problem.  Kempe, Kleinberg and Tardos deal with
this difficulty by using sampling.  In their context of approximation algorithms,
sampling is completely acceptable since it merely
creates a further approximation factor.  On the
the other hand, if one is aiming for the exact optimizer as we do, then it
is no longer apparent that sampling is an appropriate strategy.
\bibliographystyle{plain}
\bibliography{influence}

\begin{thebibliography}{10}

\bibitem{Ames}
B.~Ames.
\newblock Guaranteed clustering and biclustering via semidefinite programming.
\newblock \url{http://arxiv.org/abs/1202.3663}, 2012.

\bibitem{AmesVavasis2}
Brendan~P.W. Ames and Stephen~A. Vavasis.
\newblock Convex optimization for the planted k-disjoint-clique problem.
\newblock Submitted to Math.~Prog., 2010.

\bibitem{Bakshy}
E.~Bakshy, J.~Hofman, W.~Mason, and D.~Watts.
\newblock Everyone's an influencer: quantifying influence on {T}witter.
\newblock In {\em Proceedings of the 4th {ACM} international conference on Web
  search and data mining}, pages 65--74. ACM Press, 2011.

\bibitem{CandesTao2}
E.~Candes and T.~Tao.
\newblock Decoding by linear programming.
\newblock {\em {IEEE} Trans. Information Theory}, 51(12):4203--4215, 2005.

\bibitem{Christakis}
N.~Christakis and J.~Fowler.
\newblock The spread of obesity over a large social network over 32 years.
\newblock {\em The {N}ew {E}ngland Journal of Medicine}, 357:370--379, 2007.

\bibitem{DoanVavasis}
X.~V. Doan and S.~Vavasis.
\newblock Finding approximately rank-one submatrices with the nuclear norm and
  $\ell_1$ norm.
\newblock In review process, SIAM J.~Optimiz.;
  URL:\url{http://arxiv.org/abs/1011.1839}, 2010.

\bibitem{donoho-underdet}
D.~Donoho.
\newblock For most large underdetermined systems of linear equations the
  minimal $\ell_1$-norm solution is also the sparsest solution.
\newblock {\em Commun. Pure and Appl. Math.}, 59(6):797--829, 2006.

\bibitem{GJ79}
M.~R. Garey and D.~S. Johnson.
\newblock {\em Computers and Intractability: A Guide to the Theory of
  {NP}--completeness}.
\newblock Freeman, San Francisco, 1979.

\bibitem{Goldenberg}
J.~Goldenberg, B.~Libai, and E.~Muller.
\newblock Talk of the network: a complex systems look at the underlying process
  of word-of-mouth.
\newblock {\em Marketing Letters}, 12:211--223, 2001.

\bibitem{GrantBoyd}
M.~Grant and S.~Boyd.
\newblock Graph implementations for nonsmooth convex programs.
\newblock In V.~Blondel, S.~Boyd, and H.~Kimura, editors, {\em Recent Advances
  in Learning and Control (a tribute to {M}. {V}idyasagar)}, Lecture Notes in
  Control and Information Sciences, pages 95--110. Springer, 2008.

\bibitem{Kempe}
D.~Kempe, J.~Kleinberg, and E.~Tardos.
\newblock Maximizing the spread of influence through a social network.
\newblock In {\em Proceedings of the 9th {ACM} {SIGKDD} international
  conference on knowledge discovery and data mining}, pages 137--146, New York,
  2003. ACM Press.

\bibitem{Leskovec}
J.~Leskovec, J.~Kleinberg, and C.~Faloutsos.
\newblock Graphs over time: densification laws, shrinking diameters and
  possible explanations.
\newblock In {\em Proceedings of the 11th {ACM} {SIGKDDD} international
  conference on knowledge discovery in data mining}, pages 177--187, New York,
  2005. ACM Press.

\bibitem{sedumi}
I.~P\'olik.
\newblock Sedumi user's guide.
\newblock \url{http://sedumi.ie.lehigh.edu}, 2010.

\end{thebibliography}
\end{document}